\documentclass[a4paper,11pt,reqno]{amsart}
        \usepackage{graphicx}
        \graphicspath{ {images/} }
        \usepackage[colorlinks = true,
                    linkcolor = blue,
                    urlcolor  = blue,
                    citecolor = blue,
                    anchorcolor = blue]{hyperref}
        \usepackage{amssymb}
        \usepackage{graphicx}
        \usepackage{tikz-cd}
        \usepackage{tikz}
        \usepackage{amsmath}
        \usepackage{dsfont}
        \usepackage{shuffle}
        \usepackage{afterpage}
        \usepackage{setspace}
        \graphicspath{ {/Users/emanuele/Documents} }
        \usepackage[utf8]{inputenc}
        
        \usepackage{mathrsfs}
        \usepackage[all]{xy}
        \usetikzlibrary{knots}
        \usepackage[lite]{amsrefs}
        \usepackage{bbm}

        \newtheorem{thm}{Theorem}[section]

        \def\commutatif{\ar@{}[rd]|{\circlearrowleft}}
        
        \newcommand{\eq}[1][r]
           {\ar@<-3pt>@{-}[#1]
            \ar@<-1pt>@{}[#1]|<{}="gauche"
            \ar@<+0pt>@{}[#1]|-{}="milieu"
            \ar@<+1pt>@{}[#1]|>{}="droite"
            \ar@/^2pt/@{-}"gauche";"milieu"
            \ar@/_2pt/@{-}"milieu";"droite"}
        
        \def\dar[#1]{\ar@<2pt>[#1]\ar@<-2pt>[#1]}
          \entrymodifiers={!!<0pt,0.7ex>+} 
        
        \newcommand{\bigon}[4][r]{
            \ar@/^1pc/[#1]^{#2}_*=<0.3pt>{}="HAUT"
            \ar@/_1pc/[#1]_{#3}^*=<0.3pt>{}="BAS"
            \ar@{=>} "HAUT";"BAS" ^{#4}
          }
        
        \newcommand{\bigons}[6][r]{  
            \ar@/^2pc/[#1]^{#2}_*=<0.3pt>{}="HAUT"
            \ar@{}    [#1]     ^*=<0.3pt>{}="MILIEUHAUT"
                               _*=<0.3pt>{}="MILIEUBAS"
            \ar[#1]_(0.3){#3}                  
            \ar@/_2pc/[#1]_{#4}^*=<0.3pt>{}="BAS"
            \ar@{=>} "HAUT";"MILIEUHAUT" ^{#5}
            \ar@{=>} "MILIEUBAS";"BAS" ^{#6}
          }

        \theoremstyle{definition}
        \newtheorem{df}[thm]{Definition}

        \theoremstyle{remark}

        \newtheorem{ex}[thm]{Example}
        
        \allowdisplaybreaks

        \newcommand{\inn}{\operatorname{Inn}}
        
        \newcommand{\Z}{\mathbb{Z}}
        \newcommand{\C}{\mathbb{C}}
        \newcommand{\mR}{\mathcal{R}}

        \newcommand\rTo{\longrightarrow}

        \newcommand\mto{\longmapsto}

        \newcommand\rRack{\triangleright}
        
        \newcommand{\one}{\mathbf{1}}

        \title{Decomposition of the regular representation for dihedral quandles}

        \author{Mohamed Elhamdadi} 
        \address{Department of Mathematics, 
        University of South Florida, Tampa, FL 33620} 
        \email{emohamed@math.usf.edu} 
        
        \author{Prasad Senesi} 
        \address{Department of Mathematics, 
        The Catholic University of America,  Washington, DC 20064} 
        \email{senesi@cua.edu} 
        
        \author{Emanuele Zappala} 
        \address{Yale School of Medicine, 
        	Yale University, New Haven, CT 06511} 
        \email{emanuele.zappala@yale.edu}

\begin{document}

        \maketitle

        \begin{abstract}
        We decompose the regular quandle representation of a dihedral quandle $\mR_n$ into irreducible quandle subrepresentations.  
        \end{abstract}

        \section{Introduction}
Let $\mR_n$ be the dihedral quandle of order $n$, and let $\C \mR_n$ be the $\C$-vector space of all $\C$-valued functions on $\mR_n$.   Then $\C \mR_n$ has the structure of a quandle representation of $\mR_n$, which we call the regular representation of $\mR_n$.  In this note we provide,  for all positive integers $n$,  the decomposition of this regular representation into irreducible quandle subrepresentations.  This problem was first addressed using the notion of a \textit{quandle ring} in \cite{EFT} in 2019, although the results given there are incomplete and do not address dihedral quandles of odd order.  Our main result here was recently reported by one of the authors at the \href{https://cpb-us-e1.wpmucdn.com/blogs.gwu.edu/dist/d/4238/files/2022/04/KiW49p75abstractrev.pdf}{Knots in Washington Meeting} on April 23, 2022, and will be included in a more comprehensive forthcoming publication on the representation theory of quandles.   The authors would also like to note the recent arxiv submission \cite{KS},  where these results are obtained with a different approach for the dihedral quandles of even order.  

        \section{Review of racks and quandles}\label{review}
        \noindent We give a quick review of racks and quandles in this section.

        \begin{df}	
        \cites{EN, Joyce, Matveev} 
        A {\it rack} is a set $X$ provided with a binary operation 
        \[
        \begin{array}{lccc}
        \rRack: & X\times X & \rTo &X \\
         & (x,y) & \mto & x\rRack y
        \end{array}
        \] such that 
        \begin{itemize}
        \item[(i)] for all $x,y\in X$, there is a unique $z\in X$ such that $y=z\rRack x$;
        \item[(ii)]({\it right distributivity}) for all $x,y,z\in X$, we have $(x\rRack y)\rRack z=(x\rRack z)\rRack (y\rRack z)$.
        \end{itemize}
        \end{df}
        For $x \in X$, we denote by $R_x: X \rightarrow X$ the map given by 
        \[ R_x (y) = y \rRack x.\]
        
        Observe that property (i) also reads that for any fixed element $x\in X$, the map $R_x$ is a 
        bijection. 
        Also, notice that the right distributivity condition is equivalent to the relation $R_x(y\rRack z)=R_x(y)\rRack R_x(z)$ for all $y,z\in X$; that is, $R_x$ is an automorphism of the quandle $X$.

        Unless otherwise stated, we will always assume our racks (or quandles) to be finite racks (or quandles).

        \begin{df}
        A {\it quandle} is a rack such that 
         $x\rRack x=x,\forall x\in X$. 
        \end{df}
        
         As each $R_x$ permutes the elements of $X$, we can consider the subgroup $\left\langle R_x \right\rangle_{x \in X}$ in $S_X$ (the symmetric group on $X$) generated by the $R_x$, which we denote by $\inn(X)$ (the group of inner automorphisms of $X$).    This group acts on the set $X$,  and so the group action partitions $X$ into distinct orbits.  We say the quandle $X$ is connected if $\text{Inn}(X)$ acts transitively on $X$ (and hence there is only one orbit).  If $(X, \rRack )$ and $(Y, \rRack)$ are two quandles, a map $f: X \rightarrow Y$
        	 is a \textit{quandle homomorphism} if $f(x \rRack y) = f(x) \rRack f(y)$ for all $x, y \in X$.
        	 
        	\begin{ex}
        	Any group $G$ with the operation $x\rRack y=yxy^{-1}$ is a quandle, called conjugation quandle and denoted by Conj(G). 
        \end{ex}
        \begin{ex}
        	Any group $G$ with the operation $x\rRack y=yx^{-1}y$ is a quandle, called the Core quandle of $G$ and denoted Core(G). 
        \end{ex}
        
        \begin{ex}
        	Let $G$ be a group and $f \in {\rm Aut}(G)$.  Then one can define a quandle structure on $G$ by $x\rRack y=f(xy^{-1}) y $.  It is called a {\it generalized Alexander quandle}.	If $G$ is abelian, the operation becomes $x\rRack y=f(x)+ (Id-f) (y) $, where $Id$ stands for the identity map.  This quandle is called an {\it Alexander quandle}.
        \end{ex}
              	\begin{ex} 
       The set $\mR_n = \left\{ 1, 2, \ldots, n \right\}$, with quandle operation $x \rRack y = 2y - x$ (\text{mod } $n)$ is a quandle, called a \emph{dihedral} quandle. This is an Alexander quandle on the the group $\Z_n$ where $f$ is multiplication by $-1$.
         \end{ex}

\noindent For the remainder of this manuscript, $\mR_n$ will denote the dihedral quandle of order $n$.
        

        \section{Representation theory}
As mentioned in the introduction, the results here will be included in a later manuscript providing more results for the (complex) representation theory of (finite) quandles.  The terminology used here to describe this theory is adapted from group representation theory, a topic for which there is a wealth of literature.  There are many accessible introductions to this topic; see, for example, Chapters 4 and 5 of \cite{MR2808160}.

All vector spaces henceforth are defined over $\C$.  For any positive integer $k$, we set $\omega_k = e^{2 \pi i / k}$.  A \emph{representation} of a finite quandle $(X,*)$ on a finite dimensional complex vector space $V$ is a quandle homomorphism $\rho: X \rightarrow \text{Conj}(GL(V))$.  In other words, for all $x,y \in X$, we have $\rho(x*y)=\rho(y)\rho(x)\rho(y)^{-1}.$ To simplify the notation, we will denote $\rho(x)$ by just $\rho_x$.  Let $V$ and $W$ be two representations of a quandle $X$.  If $\rho^V: X \rightarrow \text{Conj}(GL(V))$ and $\rho^W: X \rightarrow \text{Conj}(GL(W))$ are two representations of $X$, then a linear map $\phi: V \rightarrow W$ for which the following square commutes
        \[
        \xymatrix{
        V \ar[r]^{\phi} \ar[d]_{\rho^V_x} & W \ar[d]^{\rho^W_x} \\
        V \ar[r]^{\phi} & W.
        }
        \]
         for all $x$ is called an intertwiner.  If, furthermore, $\phi$ is an isomorphism then we say that the two representations are isomorphic as representations of $X$, or \emph{equivalent}.  If $\phi: V \hookrightarrow W$ is the inclusion of the linear subspace $V$ into $W$, then we say that $V$ is a \emph{subrepresentation} of $W$. Equivalently, a subspace $U$ of $W$ is a subrepresentation of $W$ if $U$ is invariant under $\rho^W(X)$.  
        
        A representation $V$ of a quandle $X$ is called \emph{irreducible} if its only quandle subrepresentations are $\{0\}$ and $V$. A representation ia called \emph{indecompasable} if it cannot be written as direct sum of nontrivial subrepresentations.  
        
         If $\phi: \text{Inn}(X) \rightarrow \text{Aut}(V)$ is a group homomorphism, then the induced map 
   \[ \tilde{\phi}: X \rightarrow \text{Conj}(\text{Aut}(V)),  \]
   \[ \tilde{\phi}(x) : = \phi(R_x),\]
   is a quandle homomorphism. Therefore any group representation of $\text{Inn}(X)$ induces a quandle representation of $X$.

        \section{The regular representation of a quandle}  For any finite quandle $X$, we denote by $\C X $ the $\C$--vector space of $\C$-valued functions on $X$; equivalently it is the $\C$-vector space generated by basis vectors $\left\{ e_x \right\}_{ x \in X}$,  whose elements are formal sums $f = \sum_{x \in X} a_x e_x$, $a_x \in \C$.  The \textit{regular representation} of $X$ is 
        \[ \lambda: X \rightarrow \text{Conj}(GL(\C X)), \]
        where $\lambda_t(f)(x): = f(R_t^{-1}(x))$. This action of $X$ is equivalent to the right $X$--action on $\C X$ given by the linear extension of $e_x \cdot t = e_{ R_t(x) }=e_{x*t}$.    We note that a subspace $W$ of $\C X$ is a quandle subrepresentation of $
        \C X$ if and only if $W$ is a subrepresentation of $\C X$ as a \textit{group} representation of $\inn(X)$.   If a distinction is necessary, we will refer to a vector space as either a \textit{quandle representation} or a \textit{group representation}.  If two vector spaces $U, V$ are isomorphic as quandle representations of $X$, we will write $U \cong_X V$; if they are isomorphic as group representations of some group $G$, we will write $U \cong_G V$.  
        
        When $|X| \geq 2$, The regular representation always has two nontrivial subrepresentations, one of which is always irreducible: the one--dimensional subspace $\C \one$, where $\one$ is the constant function $\one (x) = 1$, for all $x \in X$. Its vector space complement, $(\C \one)^\perp$, is also a subrepresentation, since $(\C \one)^\perp = \left\{ \sum_{x \in X} a_x e_x |\; \sum a_x = 0 \right\}$, and action by any $t \in X$ only permutes the coefficients $a_x$.  So $\C X$ decomposes into a direct sum of subrepresentations 
        \[ \C X = \C \one \oplus (\C \one)^\perp.\]
        For $x, y \in X$ let $v_{xy} = e_x - e_y$.  If we enumerate $X = \left\{ x_1, \ldots, x_n \right\}$, then a basis for $(\C \one)^\perp$ is $\left\{ v_{x_1 x_2 }, v_{x_2 x_3}, \ldots, v_{x_{n-1}x_n} \right\}$.

        \section{The regular representation of a dihedral quandle}\label{Rep Dihedral}
 
The rest of this paper is devoted to the regular quandle representation of a dihedral quandle $\mR_n$.  We first explicitly describe some quandle representations of $\mR_n$.  

Since quandle subrepresentations of $\C \mR_n$ are in correspondence with group subrepresentations of $\C \mR_n$ (under the action of $\text{Inn}(\mR_n)$),  we obtain the full decomposition of $\C \mR_n$ as a quandle representation of $\mR_n$ by finding the decomposition as a group representation of $\text{Inn}(\mR_n)$. We note that this correspondence between group and quandle representations does not hold in general, but it is \emph{characteristic} of the regular representation. In general, there is no reason to expect any correspondence or similarity between quandle representations of a quandle $X$ and group representations of $\text{Inn}(X)$.

By virtue of the aforementioned correspondence, and the fact that the group $\text{Inn}(\mR_n)$ is a dihedral group, we will consider in detail the representation theory of dihedral groups.
We will let $D_m$ be the dihedral group of order $2m$.  Then it is shown in \cite{EMR} that 
\[ \text{Inn}(\mR_n) \cong \begin{cases} D_{n/2}, & n \text{ even}, \\
D_n, & n \text{ odd}.\end{cases}
\]

We describe the (isomorphism classes of) finite--dimensional irreducible group representations of the dihedral group $D_m$ here, to facilitate statement of results below.  The dihedral group $D_m$ has a presentation given by 
        \[ D_m = \left\langle \alpha, \beta \; \; | \; \; |\alpha| = |\beta| = 2, \; |\alpha \beta | = m \right\rangle. \]
        We fix generators $\alpha, \beta$ for this presentation.  The classification of the finite--dimensional irreducible group representations of $D_m$ falls into two cases: $m$ even and $m$ odd.  
        If $\lambda, \mu \in \C$, we will denote by $\C(\lambda, \mu)$ the one--dimensional group representation of $D_m$ on which $\alpha$ and $\beta$ act by scalar multiplication by $\lambda$ and $\mu$, respectively.  And for nonzero $\lambda \in \C$, we will denote by $W(\lambda)$ the two--dimensional group representation of $D_m$ for which the matrix representations of $\alpha$ and $\beta$ are 
        \[ \alpha \mapsto \begin{bmatrix}
        0 & 1 \\ 
        1 & 0 
        \end{bmatrix}, \; \; \; \; \beta \mapsto \begin{bmatrix}
        0 & \lambda \\
        \lambda^{-1} & 0 
        \end{bmatrix}.
        \]
        With this notation fixed, we provide the classification of the finite--dimensional irreducible group representations of $D_m$.   All such group representations are either 1-- or 2-- dimensional, and are listed here:
        
        \begin{center}
        \begin{tabular}{|c|c|c|}
        \hline 
        $D_m$ & $m$ even & $m$ odd \\
        \hline
        & & \\
        1-dimensional  & $\C(1,1)$, $\C(1,-1)$, & \\ group representations &  
           $\C(-1,1)$, $\C(-1,-1) $& $\C(1,1)$, $\C(-1,-1)$\\
        && \\
        \hline
        &&\\
        2-dimensional &$W(\omega_m^s)$ , & $W(\omega_m^s)$ , \\ group representations &   $1 \leq s \leq m/2-1 $ &   $1 \leq s \leq (m-1)/2$ \\
        &&\\
        \hline
        \end{tabular}
        \end{center}
        We will denote by $\Gamma(D_m)$ the set of (isomorphism classes of) finite--dimensional irreducible group representations of $D_m$.  For example, $D_8$ has 4 representations of dimension 1 and 4 irreducible representations of dimension 2, and they are 
        \[ \Gamma(D_{8}) = \left\{ \C( \pm 1,\pm 1),  \; \; \C( \pm 1, \mp 1), \; \; W(\omega_8), \; \; W(\omega_8^2), \; \; W(\omega_8^3) \right\}, \]
        while $D_9$ has 2 representations of dimension 1 and 3 irreducible representations of dimension 2, and they are 
        \[ \Gamma(D_{9}) = \left\{ \C( \pm 1,\pm 1),  \; \; W(\omega_9), \; \; W(\omega_9^2), \; \; W(\omega_9^3), \; \; W(\omega_9^4) \right\}. \]

Since $\text{Inn}(\mR_n)$ is a dihedral group generated by $R_1$ and $R_2$, these group representations also provide us with examples of 1-dimensional and 2-dimensional irreducible quandle representations of $\mR_n$.  We will make this precise as follows: let $\C(\lambda,  \mu)$ be the 1-dimensional representation of $\text{Inn}(\mR_n)$ for which $R_1$ and $R_2$ act by scalar multiplication by $\lambda$ and $\mu$, respectively, and let $W(\lambda)$ be the two-dimensional representation of $\text{Inn}(\mR_n)$ for which $R_1$ and $R_2$ have matrix representations 
     \[ R_1 \mapsto \begin{bmatrix}
        0 & 1 \\ 
        1 & 0 
        \end{bmatrix}, \; \; \; \; R_2 \mapsto \begin{bmatrix}
        0 & \lambda \\
        \lambda^{-1} & 0 
        \end{bmatrix}.
        \]

        Now we consider the regular quandle representation $\C \mR_n $ of $\mR_n$ and the subrepresentation $(\C \one)^\perp$, in cases $n$ even or $n$ odd. In either case, $(\C \one)^\perp$ is spanned by $\left\{ v_{ij} \right\}_{1 \leq i, j \leq n }$. 
        
        If $n = 2r$ is even,  $\mR_{n}$ has two orbits: the even orbit $\left\{ 2, 4, \ldots, 2r \right\}$, and the odd orbit $\left\{ 1, 3, \ldots, 2r-1 \right\}$.  Let 
        \[ \Phi_{n,0} = \text{sp}\left\{ v_{ij} \right\}_{i \neq j \text{ even}}, \; \; \; \Delta_{n,0} = \left\{ v_{24}, v_{46}, \ldots, v_{2r-2 \; 2r} \right\}, \]
         and 
         \[ \Phi_{n,1} = \text{sp}\left\{ v_{ij} \right\}_{i \neq j \text{ odd}}, \; \; \; \Delta_{n,1} = \left\{ v_{13}, v_{35}, \ldots, v_{2r-3 \; 2r-1} \right\}. \]
        We call $\Phi_{n,0}$ and $\Phi_{n,1}$ the even and odd subspaces, respectively, of $(\C \one)^\perp$.   For $i = 0, 1$, $\Delta_{n,i}$ is a basis for $\Phi_{n,i}$.
        \\
        
        If $n$ is odd, $\mR_n$ has a single orbit.  In this case, we set 
         \[ \Phi_n= \text{sp}\left\{ v_{ij} \right\}_{i \neq j}, \; \; \; \Delta_n = \left\{ v_{12}, v_{23}, \ldots, v_{n-1 \; n} \right\}, \]
         and $\Delta_n$ is a basis for $\Phi_n$.  Because the orbits of $\mR_n$ are as described above, all of the subspaces $\Phi_{n, i}$, $\Phi_n$ are quandle subrepresentations of $\C \mR_n$.
        
        \begin{thm}\label{dihedral_decomp} If $n$ is even, then the decomposition of the regular representation $\C\mR_n$ of $\mR_n$ into irreducible quandle subrepresentations  is, for $n \equiv 0  \bmod  4  $, 
          \[ \C\mR_n  \cong_{\mR_n}  \C( 1, 1)^{\oplus 2} \oplus \C( \pm 1, \mp 1) \bigoplus_{s=1}^{  n/4 - 1  } W(\omega_{n/2}^s)^{\oplus 2}, \]
        \hspace{.1 in}  and for  $n \equiv 2  \bmod  4  $,  \[ \C\mR_n  \cong_{\mR_n} \C( 1, 1)^{\oplus 2} \bigoplus_{s=1}^{ (n-2)/4 } W(\omega_{n/2}^s)^{\oplus 2}. \]
        \\
        If $n$ is odd,  then the decomposition of $\C\mR_n$ into irreducible quandle subrepresentations of $\mR_n$ is 
        \[ \C\mR_n  \cong_{\mR_n} \C (1,1) \bigoplus_{s = 1}^{(n-1)/2} W(\omega_n^{2s}). \]
        \end{thm}
        \begin{proof} For arbitrary $n$, let $\hat{\mathbf{1}} =\sum_{i=1}^n (-1)^{i}e_i \in ( \C \mathbf{1})^\perp$.  Then, as already pointed out before, $\C\mR_n$ decomposes as $ \C\mR_n = \C \mathbf{1} \oplus (\C \mathbf{1})^\perp$, and $\C \mathbf{1} \cong_{\mR_n} \C \hat{\mathbf{1}} \cong_{\mR_n} \C(1,1)$.  We proceed distinguishing the cases $n$ even or $n$ odd. 
        \\
        \\
        \noindent \textbf{Case 1:} $n$ is even.  Let $n= 2r$.  In this case, each $\Phi_{n,i}$ (for $i=0,1)$ is a subrepresentation of $\C\mR_{n}$ of dimension $r-1$, and $(\C \mathbf{1})^\perp$ decomposes into $\mR_n$-subrepresentations as 
        \[ (\C \mathbf{1})^\perp = \C \hat{\mathbf{1}} \oplus \Phi_{n,0} \oplus \Phi_{n,1}. \]
        Now set $[ a_1, \ldots, a_{r-1} ]_0 = \displaystyle{\sum_{i=1}^{r-1}{a_i}v_{2i, 2i+2}}$ (the coordinate vector in the basis $\Delta_{n,0}$).  The right multiplication operators $R_1$ and $R_2$ together generate all of $\text{Inn}(\mR_n) \cong D_{n/2}$, and with respect to the basis $\Delta_{n,0}$, the matrix representations of these operators on $\Phi_{n,0}$ are 
        \begin{equation}\label{matrixrep}
         \left[ R_1 \right]_{\Delta_{n, 0}} = \begin{bmatrix}
        0 & 0 & \cdots & 0 &-1 \\
        0 & 0 & \cdots & -1 & 0 \\
        \vdots & \vdots & \reflectbox{$\ddots$}  & \vdots & \vdots \\
        0 & -1 & \cdots & 0 & 0 \\
        -1 & 0 & \cdots & 0 & 0
        \end{bmatrix}
        , \; \; \; \;  \left[ R_2 \right]_{\Delta_{n, 0}} = \begin{bmatrix}
        1 & 0 & \cdots & 0 & 0 \\
        1 & 0 & \cdots &  0 & -1 \\
        \vdots & \vdots & \reflectbox{$\ddots$}  & \vdots & \vdots \\
        1 & 0 & \cdots & 0 & 0 \\
        1 & -1 & \cdots & 0 & 0
        \end{bmatrix}.
        \end{equation}
        
        \pagebreak 
        
        For example, for $n = 12$, we have 
        \[ R_1\left(  \left[ a_1, a_2, a_3, a_4, a_5 \right]_0 \right) = \left[  -a_5, -a_4, -a_3, -a_2, -a_1 \right]_0,  \]
        \[ R_2\left(  \left[ a_1, a_2, a_3, a_4, a_5 \right]_0 \right) = \left[  a_1, a_1 -a_5, a_1 - a_4, a_1 -a_3, a_1 - a_2 \right]_0. \]

        \noindent  For $1 \leq s \leq r-1$, let 
        \[ \mathbf{u}_s  :=  \left[ 1-\omega_r^s, 1 - (\omega_r^s)^2, \cdots, 1 - (\omega_r^s)^{r-1} \right]_0 = \sum_{i=1}^{r-1} \left(1-(\omega_r^s)^i\right) v_{2i, 2i+2}, \;  \text{ and } \;  \mathbf{v}_s := R_1(\mathbf{u}_s).\]
        Then $\mathbf{v}_s = \displaystyle{\sum_{i=1}^{r-1} \left((\omega_r^s)^{r-i}-1 \right) v_{2i, 2i+2}}$, and   
        \begin{eqnarray*}
        R_2(\mathbf{u}_s) &=& \sum_{i=1}^{r-1} \left((\omega_r^s)^{r+1-i}-\omega_r^s \right) v_{2i, 2i+2}\\
        &=& \omega_r^s  \sum_{i=1}^{r-1} \left((\omega_r^s)^{r-i}-1 \right) v_{2i, 2i+2}\\
        &=& \omega_r^s \mathbf{v}_s.
        \end{eqnarray*}
        Since all $R_i$ are involutions, we also have $R_2(\mathbf{v}_s) = (\omega_r^{-1})^s \mathbf{u}_s$. Let $U_{s,0}$ be the subspace of $\Phi_{n,0}$ spanned by $\mathbf{u}_s$ and $\mathbf{v}_s$.  The above identities show that each $U_{s,0}$ is a quandle subrepresentation of $\Phi_{n,0}$.  If $\mathbf{u}_s$ and $\mathbf{v}_s$ are linearly dependent, then  $U_{s,0}$ is a 1-dimensional quandle subrepresentation.  And if $\mathbf{u}_s$ and $\mathbf{v}_s$ are linearly independent, the matrix representations of $R_1$ and $R_2$ with respect to the basis $\left\{ \mathbf{u}_s, \mathbf{v}_s \right\}$ are
        \[  \left[ R_1 \right]_{\left\{ \mathbf{u}_s, \mathbf{v}_s \right\} }= \begin{bmatrix}
        0 & 1 \\ 
        1 & 0 
        \end{bmatrix}, \; \; \; \; \left[ R_2 \right]_{\left\{ \mathbf{u}_s, \mathbf{v}_s \right\} } = \begin{bmatrix}
        0 & \omega_r^s \\
        \omega_r^{-s} & 0 
        \end{bmatrix},\]
        which gives us $U_{s,0} \cong W(\omega_r^s)$ (both an $\mR_n$-quandle and an $\text{Inn}(\mR_n)$-group  isomorphism).   From the identity $(\omega_r^{r-s})^\ell = (\omega_r^s)^{-\ell}$ (for any integers $\ell, s$),  we obtain $\mathbf{u}_s = - \mathbf{v}_{r-s}$, for $1 \leq s \leq r-1$.  Therefore $U_{s,0} = U_{r-s, 0}$, $1 \leq s \leq r-1$.    And since each $W(\omega_r^s)$ is an irreducible group representation for $\text{Inn}(\mR_n)$ for $1 \leq s \leq r-1$,  it follows that $U_{s,0} \cong_{\mR_n} U_{v,0}$ if and only if $s = v$, for $1 \leq s, v \leq \lfloor r/2 \rfloor$, and $U_{s,0} \cap U_{v,0} = \left\{ 0 \right\}$ for $1 \leq s \neq v \leq \lfloor r/2 \rfloor$. 
 We summarize these results here: \\
        
        \noindent For $r$  even,  $U_{\ell,0} \cap U_{s,0} = \left\{ 0 \right\}$ for $1 \leq \ell \neq s \leq r/2$, and $\displaystyle{\dim(U_{s,0}) = \begin{cases}
        2, & 1 \leq s < r/2\\
        1, & s = r/2
        \end{cases} }$ .  \\ \\ 
        For $r$ odd, $U_{\ell,0} \cap U_{s,0} = \left\{ 0 \right\}$ for $1 \leq \ell \neq s \leq (r-1)/2$, and $\dim(U_{s,0}) = 2$ for all~$s$.
        \\
        \\ 
        Hence we obtain a decomposition of $\Phi_{n,0}$ into irreducible quandle subrepresentations:
        \[   \Phi_{n,0} \cong_{\mR_n} \bigoplus_{s=1}^{\lfloor r/2 \rfloor } U_{s,0} \cong_{\mR_n} \begin{cases}  W(\omega_r)\oplus \cdots \oplus W\left(\omega_r^{r/2-1}\right) \oplus \C(-1, 1) , & r \; \text{even}, \\
         W(\omega_r)\oplus \cdots \oplus W\left(\omega_r^{(r-1)/2}\right) , & r \; \text{odd}. \end{cases}  \]
        The decomposition of $\Phi_{n,1}$ is achieved in an similar fashion - with a few small changes.  We first replace $R_1$, $R_2$ with $R_{n/2}$ and $R_1$, respectively.  The matrix representations of these transformations, with respect to the basis $\Delta_{n,1}$, are given by 
        
        \begin{equation}\label{matrixrep2}
         \left[ R_{n/2} \right]_{\Delta_{n, 1}} = \begin{bmatrix}
        0 & 0 & \cdots & 0 &-1 \\
        0 & 0 & \cdots & -1 & 0 \\
        \vdots & \vdots & \reflectbox{$\ddots$}  & \vdots & \vdots \\
        0 & -1 & \cdots & 0 & 0 \\
        -1 & 0 & \cdots & 0 & 0
        \end{bmatrix}
        , \; \; \; \;  \left[ R_1 \right]_{\Delta_{n, 1}} = \begin{bmatrix}
        1 & 0 & \cdots & 0 & 0 \\
        1 & 0 & \cdots &  0 & -1 \\
        \vdots & \vdots & \reflectbox{$\ddots$}  & \vdots & \vdots \\
        1 & 0 & \cdots & 0 & 0 \\
        1 & -1 & \cdots & 0 & 0
        \end{bmatrix}.
        \end{equation}

        \noindent Then we define 
        \begin{eqnarray*}
        \mathbf{u}_s  &:=&  \left[ 1-\omega_r^s, 1 - (\omega_r^s)^2, \cdots, 1 - (\omega_r^s)^{r-1} \right]_1 \\
        &=& \sum_{i=1}^{r-1} \left(1-(\omega_r^s)^i\right) v_{2i-1, 2i+1}, \; \; \; \; \; \text{ and } \;  \mathbf{v}_s := \omega_r^{-s} R_{n/2}(\mathbf{u}_s).
        \end{eqnarray*}
Then we can show that $R_1(\mathbf{u}_s) = \mathbf{v}_s$, and since 
\[ R_2 = R_{n/2 +2} = R_{n/2\, \rRack 1} = R_1 R_{n/2} R_1^{-1} = R_1 R_{n/2} R_1,  \]
we have $R_2(\mathbf{u}_s) = R_1 R_{n/2} R_1(\mathbf{u}_s) =  \omega_r^{s} \mathbf{v}_s$. For $1 \leq s \leq r-1$,  we define $U_{s,1}$ to be the subspace of $\Phi_{n,1}$ spanned by $\mathbf{u}_s$ and $\mathbf{v}_s$.  Then, as above,  we have 
\[ \dim(U_{s,1} )= \begin{cases} 1, & s = r/2, \, \, r \text{ even} \\
 2,  &  s \neq r/2, \end{cases} \]
and for $s \neq r/2$, the matrix representations of $R_{1}$ and $R_2$ are 
 \[  \left[ R_{1} \right]_{\left\{ \mathbf{u}_s, \mathbf{v}_s \right\} }= \begin{bmatrix}
        0 & 1 \\ 
        1 & 0 
        \end{bmatrix}, \; \; \; \; \left[ R_2 \right]_{\left\{ \mathbf{u}_s, \mathbf{v}_s \right\} } = \begin{bmatrix}
        0 & \omega_r^s \\
        \omega_r^{-s} & 0 
        \end{bmatrix}.\]
This shows 
                    \[   \Phi_{n,1} \cong \bigoplus_{s=1}^{\lfloor r/2 \rfloor } U_{s,1}\cong_{\mR_n} \begin{cases}  W(\omega_r)\oplus \cdots \oplus W\left(\omega_r^{r/2-1}\right) \oplus \C(1, -1) , & r \; \text{even}, \\
         W(\omega_r)\oplus \cdots \oplus W\left(\omega_r^{(r-1)/2}\right) , & r \; \text{odd}. \end{cases}\]

        For $r$ even ($ n \equiv 0 \bmod 4$), the one-dimensional quandle subrepresentations of $\C\mR_{n}$ are:
        \[ \C \mathbf{1} \cong \C(1,1), \; \; \; \; \C \hat{\mathbf{1}} \cong \C(1,1), \; \; \; U_{r/2,0} \cong \C(-1,1), \; \; \;  U_{r/2,1} \cong \C(1,-1), \]
        and the two-dimensional irreducible quandle subrepresentations are $U_{s,0} \cong U_{s,1} \cong W(\omega_r^s)$,  $1 \leq s \leq r/2-1$.
        
        For $r$ odd ($n \equiv 2 \bmod 4)$,  the one-dimensional quandle subrepresentations of $\C\mR_{n}$ are:
        \[ \C \mathbf{1} \cong \C(1,1)  \; \;  \text{ and }\; \; \C \hat{\mathbf{1}} \cong \C(1, 1) ,\]
        and the two-dimensional irreducible quandle subrepresentations are $U_{s,0} \cong U_{s,1} \cong W(\omega_r^s)$,  $1 \leq s \leq (r~-~1)~/~2$.
       \\
       \\
 \noindent \textbf{Case 2:} $n$ odd. Let $n= 2r+1$.  With respect to the basis $\Delta_n = \left\{ v_{12}, v_{23}, \ldots, v_{n-1, n} \right\}$, the elements $R_1$, $R_2 \in \inn(\mR_n)$ have matrix representations
        \begin{equation*}
         \left[ R_1 \right]_{\Delta_n} = \begin{bmatrix}
        0  & \cdots & 0 & -1 & 1 \\
        0  & \cdots  & -1 & 0  & 1 \\
        \vdots  & \reflectbox{$\ddots$}  & \vdots & \vdots & \vdots \\
        -1 & 0 & \cdots & 0 & 1 \\
        0 & 0 & \cdots & 0 & 1
        \end{bmatrix}
        , \; \; \; \;  \left[ R_2 \right]_{\Delta_n} = \begin{bmatrix}
        1 & 0 & \cdots & 0 & 0 \\
        1 & 0 & \cdots &  0 & -1 \\
        \vdots & \vdots & \reflectbox{$\ddots$}  & \vdots & \vdots \\
        1 & 0 & \cdots & 0 & 0 \\
        1 & -1 & \cdots & 0 & 0
        \end{bmatrix}.
        \end{equation*}
         For $1 \leq s \leq n-1$, let 
        \[ \mathbf{u}_s  =  \left[ 1-\omega_n^s, 1 - (\omega_n^s)^2, \cdots, 1 - (\omega_n^s)^{n-1} \right], \; \; \; \text{ and } \; \; \mathbf{v}_s = R_1(\mathbf{u}_s).\]
        We can then show that $\mathbf{v}_s = \omega_n^{-s}\left( \omega_n^{-s} - 1, (\omega_n^{-s})^2 - 1, \ldots, (\omega_n^{-s})^{n-1} - 1\right)$.

        \noindent Let $U_{s}$ be the subspace of $\Phi_{n}$ spanned by $\mathbf{u}_s$ and $\mathbf{v}_s$.   We can show that $\mathbf{u}_s = \mathbf{v}_{n-s}$, hence $U_s = U_{n-s}$.  Furthermore, we have 
        \begin{eqnarray*}
        R_2(\mathbf{u}_s) &=& R_2( 1 - \omega_n^s, 1 - (\omega_n^s)^2, \ldots, 1 - (\omega_n^s)^{n-1})\\
        &=& ( 1 - \omega_n^s, (\omega_n^s)^{n-1} - \omega_n^s, (\omega_n^s)^{n-2} - \omega_n^s, \ldots, (\omega_n^s)^{2} - \omega_n^s) \\
         &=& \omega_n^{2s} \omega_n^{-s} ( \omega_n^{-s} -1, (\omega_n^{-s})^2 -1, \ldots, (\omega_n^{-s})^{n-1} - 1)\\
         &=& \omega_n^{2s} \mathbf{v}_s.
        \end{eqnarray*}
        Therefore, if $\mathbf{u}_s$ and $\mathbf{v}_s$ are linearly independent, the matrix representations of $R_1$ and $R_2$ with respect to the basis $\left\{ \mathbf{u}_s, \mathbf{v}_s \right\}$ of $U_s$ are again
        \[  \left[ R_1 \right]_{\left\{ \mathbf{u}_s, \mathbf{v}_s \right\} }= \begin{bmatrix}
        0 & 1 \\ 
        1 & 0 
        \end{bmatrix}, \; \; \; \; \left[ R_2 \right]_{\left\{ \mathbf{u}_s, \mathbf{v}_s \right\} } = \begin{bmatrix}
        0 & \omega_n^{2s} \\
        \omega_n^{-2s} & 0 
        \end{bmatrix},\]
 We can show linear independence of $\mathbf{u}_s$ and $\mathbf{v}_s$ for $1 \leq s \leq r$, which gives us the decomposition 
 \[  (\C \mathbf{1})^\perp = \bigoplus_{s=1}^r U_s, \] 
 and we obtain the desired decomposition into irreducible representations 
        \[ \C\mR_n  \cong \C (1,1) \bigoplus_{s = 1}^{r} W(\omega_n^{2s}). \]
        \end{proof}
        
        
\subsection*{Examples}
        
        The decomposition of $\C \mR_n$ given in Theorem \ref{dihedral_decomp} falls into 3 cases: $n \equiv 0 \bmod 4$, $n \equiv 2 \bmod 4$, and $n$ odd.  We provide three examples below, one for each of these cases, wherein we provide the decomposition and all irreducible subrepresentations of $\C \mR_{10}$, $\C \mR_{11}$, and $\C \mR_{12}$:
        
        \vspace{0.2 in}
        
        \begin{center}
        Regular quandle representation of $\mR_{10}:$
        \[ \C \mR_{10} = \C \mathbf{1}  \oplus \overbrace{\C \hat{\mathbf{1}} \oplus \underbrace{ U_{1,0} \oplus U_{2,0}}_{\Phi_{10,0}}\oplus\underbrace{ U_{1,1} \oplus U_{2,1}}_{\Phi_{10,1}} }^{(\C \mathbf{1})^{\perp}}\]
        
        \vspace{.1 in}
        
        \begin{tabular}{clllcl}
        subrep of $\C \mR_{10}$ & & irrep of $D_5$ & generated by   & dimension\\[4pt]
        \hline\\[-4pt]
         $\C \mathbf{1}$ & $\cong$ &  $\C(1,1) $ & $\hspace{.5in} \mathbf{1}$ & 1\\[4pt]
        $\C \hat{\mathbf{1}} $ & $ \cong $ &$\C(1,1)$ &  \hspace{.45in} $\hat{\mathbf{1}}$ & 1 \\[4pt]
        $U_{1,0}$ & $ \cong $ & $W(\omega_5) $ & $\sum_{i=1}^4 \left( 1 - \omega_5^i
        \right)v_{2i, 2i+2}$ \hspace{.37in} &  2 \\[4pt]
        $U_{2,0}$ & $ \cong $ & $W(\omega_5^2 ) $ & $\sum_{i=1}^4 \left( 1 - \omega_5^{2i}\right)v_{2i, 2i+2}$ & 2 \\[4pt]
        $U_{1,1}$ & $ \cong $ & $W(\omega_5) $& $\sum_{i=1}^4 \left( 1 - \omega_5^{i}\right)v_{2i-1, 2i+1}$   & 2\\[4pt]
        $U_{2,1}$ & $ \cong $ & $W(\omega_5^2 ) $& $\sum_{i=1}^4 \left( 1 - \omega_5^{2i}\right)v_{2i-1, 2i+1}$ & 2 \\[4pt]
        \end{tabular}
        \end{center}
    
\pagebreak 

\mbox{} 

\vspace{0.4 in}
    
        \begin{center}
        
          Regular quandle representation of $\mR_{11}:$
        \[ \C \mR_{11} = \C \mathbf{1}  \oplus \overbrace{  U_1 \oplus U_{2} \oplus U_{3} \oplus U_4 \oplus U_5 }^{(\C \mathbf{1})^{\perp}}\]
        
        \vspace{.1 in}
        
        \begin{tabular}{clllcl}
        subrep of $\C \mR_{11}$ & & irrep of $D_{11}$ & generated by   & dimension\\[4pt]
        \hline\\[-4pt]
        $\C \mathbf{1}$ & $\cong$ &  $\C(1,1) $ &  \hspace{.5in} $ \mathbf{1}$ & 1\\[4pt]
        $U_{1}$ & $ \cong $ & $W(\omega_{11}^2) $&  $ \sum_{i=1}^8 \left( 1 - \omega_{11}^{2i}\right)e_i$ \hspace{.58in} & 2 \\[4pt]
        $U_{2}$ & $ \cong $ & $W(\omega_{11}^4) $&  $\sum_{i=1}^8 \left( 1 - \omega_{11}^{4i}\right)e_i$& 2 \\[4pt]
        $U_{3}$ & $ \cong $ & $W(\omega_{11}^6) $&  $\sum_{i=1}^8 \left( 1 - \omega_{11}^{6i}\right)e_i$& 2 \\[4pt]
        $U_{4}$ & $ \cong $ & $W(\omega_{11}^8) $&  $\sum_{i=1}^8 \left( 1 - \omega_{11}^{8i}\right)e_i$& 2 \\[4pt]
        $U_{5}$ & $ \cong $ & $W(\omega_{11}^{10}) $&  $\sum_{i=1}^8 \left( 1 - \omega_{11}^{10i}\right)e_i$& 2 \\[4pt]
        \end{tabular}
        \end{center}

        \vfill 
        
        \begin{center}
        
        Regular quandle representation of $\mR_{12}:$
        \[ \C \mR_{12} = \C \mathbf{1}  \oplus \overbrace{\C \hat{\mathbf{1}} \oplus \underbrace{ U_{1,0} \oplus U_{2,0} \oplus U_{3,0}}_{\Phi_{12,0}}\oplus\underbrace{ U_{1,1} \oplus U_{2,1} \oplus U_{3,1}}_{\Phi_{12,1}} }^{(\C \mathbf{1})^{\perp}}\]
        
        \vspace{.1 in}
        
        \begin{tabular}{clllcl}
        subrep of $\C \mR_{12}$ & & irrep of $D_6$ & generated by   & dimension\\[4pt]
        \hline\\[-4pt]
        $\C \mathbf{1}$ & $\cong$ &  $\C(1,1) $ & $\hspace{.5in} \mathbf{1}$ & 1\\[4pt]
        $\C \hat{\mathbf{1}} $ & $ \cong $ &$\C(1,1)$ &  \hspace{.45in} $\hat{\mathbf{1}}$ & 1 \\[4pt]
        $U_{1,0}$ & $ \cong $ & $W(\omega_6) $ & $\sum_{i=1}^5 \left( 1 - \omega_6^i
        \right)v_{2i, 2i+2}$ &  2 \\[4pt]
        $U_{2,0}$ & $ \cong $ & $W(\omega_6^2 ) $ & $\sum_{i=1}^5 \left( 1 - \omega_6^{2i}\right)v_{2i, 2i+2}$ & 2 \\[4pt]
        $U_{3,0}$ & $ \cong $ & $\C(-1,1)$ & $\sum_{i=1}^5 \left( 1 - 
        (-1)^{i}\right)v_{2i, 2i+2}$   & 1 \\[4pt]
        $U_{1,1}$ & $ \cong $ & $W(\omega_6) $& $\sum_{i=1}^5 \left( 1 - \omega_6^i\right)v_{2i-1, 2i+1}$   & 2\\[4pt]
        $U_{2,1}$ & $ \cong $ & $W(\omega_6^2 ) $& $\sum_{i=1}^5 \left( 1 - \omega_6^{2i}\right)v_{2i-1, 2i+1}$ & 2 \\[4pt]
        $U_{3,1}$ & $ \cong $ & $\C(1,-1) $& $\sum_{i=1}^5 \left( 1 - (-1)^i \right)v_{2i-1, 2i+1}$  & 1\\[4pt]
        \end{tabular}
        \end{center}

\vfill 

\pagebreak


\begin{thebibliography}{0}
        	
        	
        
       	
        	\bib{EMR}{article}{
           author={Elhamdadi, Mohamed},
           author={Macquarrie, Jennifer},
           author={Restrepo, Ricardo},
           title={Automorphism groups of quandles},
           journal={J. Algebra Appl.},
           volume={11},
           date={2012},
           number={1},
           pages={1250008, 9},
           issn={0219-4988},
           review={\MR{2900878}},
           doi={10.1142/S0219498812500089},
        }
        	
        	
        		\bib{EFT}{article}{
           author={Elhamdadi, Mohamed},
           author={Fernando, Neranga},
           author={Tsvelikhovskiy, Boris},
           title={Ring theoretic aspects of quandles},
           journal={J. Algebra},
           volume={526},
           date={2019},
           number={15},
           pages={166-187},
           review={\MR{3915329}},
        }
        
        
        
        	
        	\bib{EN}{book}{
        		author={Elhamdadi, Mohamed},
        		author={Nelson, Sam},
        		title={Quandles---an introduction to the algebra of knots},
        		series={Student Mathematical Library},
        		volume={74},
        		publisher={American Mathematical Society, Providence, RI},
        		date={2015},
        		pages={x+245},
        		isbn={978-1-4704-2213-4},
        		review={\MR{3379534}},
        	}
        	
       
        	
        	\bib{MR2808160}{book}{
        		author={Etingof, Pavel},
        		author={Golberg, Oleg},
        		author={Hensel, Sebastian},
        		author={Liu, Tiankai},
        		author={Schwendner, Alex},
        		author={Vaintrob, Dmitry},
        		author={Yudovina, Elena},
        		title={Introduction to representation theory},
        		series={Student Mathematical Library},
        		volume={59},
        		note={With historical interludes by Slava Gerovitch},
        		publisher={American Mathematical Society, Providence, RI},
        		date={2011},
        		pages={viii+228},
        		isbn={978-0-8218-5351-1},
        		review={\MR{2808160}},
        		doi={10.1090/stml/059},
        	}
        	
        	\bib{Joyce}{article}{
        		author={Joyce, D.},
        		title={A classifying invariant of knots, the knot quandle},
        		journal={J. Pure Appl. Algebra},
        		volume={23},
        		date={1982},
        		number={1},
        		pages={37--65},
        		issn={0022-4049},
        		review={\MR {638121}},
        		doi={10.1016/0022-4049(82)90077-9},
        	}
        	
        
        	
        	\bib{Matveev}{article}{
        		author={Matveev, S. V.},
        		title={Distributive groupoids in knot theory},
        		language={Russian},
        		journal={Mat. Sb. (N.S.)},
        		volume={119(161)},
        		date={1982},
        		number={1},
        		pages={78--88, 160},
        		issn={0368-8666},
        		review={\MR {672410}},
        	}
        
       
        
            \bib{KS}{article}{
        			author={Kaur, Dilpreet},
        			author={Singh,  Pushpendra}
        			title={Decomposition of quandle rings of dihedral quandles},
        			status={arXiv:2206.03162, 2022 }	
        		}

        	
        	
        	
        \end{thebibliography}
        \end{document}